\documentclass[12pt,letterpaper]{article}
\usepackage{graphicx}
\usepackage{amsmath}
\usepackage{amsfonts}
\usepackage{amsthm}
\usepackage{amssymb}
\usepackage{color}

\usepackage{verbatim} 
\usepackage{hyperref}
\usepackage{breakurl}

\newtheorem{prop}{Proposition}

\newtheorem{thm}[prop]{Theorem}
\newtheorem{lemma}[prop]{Lemma}

\theoremstyle{remark}

\newcommand{\be}{\begin{equation}}
\newcommand{\ee}{\end{equation}}
\newcommand{\bea}{\begin{eqnarray}}
\newcommand{\eea}{\end{eqnarray}}
\newcommand{\beas}{\begin{eqnarray*}}
\newcommand{\eeas}{\end{eqnarray*}}

\begin{document}

\title{Skyscraper Numbers}

\author{Tanya Khovanova\\MIT \and Joel Brewster Lewis\\UMN}
\maketitle

\begin{abstract}
We introduce numbers depending on three parameters which we call \emph{skyscraper numbers}. We discuss properties of these numbers and their relationship with Stirling numbers of the first kind, and we also introduce a skyscraper sequence.
\end{abstract}

\section{Introduction}

In skyscraper puzzles you have to put an integer from $1$ to $n$ in
each cell of a square grid. Integers represent heights of buildings.
Every row and column needs to be filled with buildings of different
heights and the numbers outside the grid indicate how many buildings you
can see from this direction. For example, in the sequence $213645$ you can
see three buildings from the left (2,3,6) and two from the right (5,6).

In mathematical terminology, we are asked to build a Latin square such
that each row is a permutation of length $n$ with a given number of
left-to-right and right-to-left-maxima. The $7$-by-$7$ skyscraper puzzle in Figure~\ref{fig:puzzle} is from the Eighth World Puzzle Championship.

\begin{figure}[htp]
\centering
\includegraphics[scale=0.6]{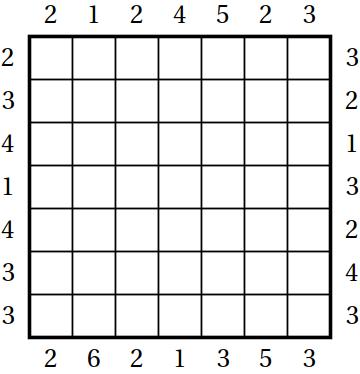}
\caption{A skyscraper puzzle}
\label{fig:puzzle}
\end{figure}

Latin squares are notoriously complicated and difficult to enumerate.
Thus, in this paper we discuss the combinatorics of a single row instead of asking about the entire puzzle.  We give formal definitions and provide examples in Section~\ref{sec:definitions}. In Section~\ref{sec:numbers} we give formulas for skyscraper numbers and in Section~\ref{sec:sequence} we define the skyscraper sequence. In Section~\ref{sec:properties} we prove some properties of skyscraper numbers and the sequence.

\section{Definitions and Examples}\label{sec:definitions}

What can you say about a row of a skyscraper puzzle if you ignore all information except the numbers given at the two ends of the row?  In particular, how many permutations of the buildings are there that satisfy the given restrictions?  Of course, the outside numbers have to be between $1$ and $n$, and we leave as an exercise the proof that their sum must be between $3$ and $n+1$.  Now we consider some of the simplest cases.

Suppose the two given numbers are
$n$ and $1$. In this case, the row is completely defined. There is
only one possibility: the buildings should be arranged in the increasing
order from the side where we see all of them.

Suppose the grid size is $n$ and the outside numbers are $a$ and $b$. Let us
denote the total number of permutations by $f_{n}(a, b)$. We
will assume that $a$ is on the left and $b$ is on the right. We call these number $f_{n}(a, b)$ the \textit{skyscraper numbers} and they are the object of our study.

In the previous example, we showed that $f_n(n, 1) = 1$. And of
course we have $f_n(a, b) = f_{n}(b, a)$.

Let us discuss a couple of other examples.

First, we want to discuss the case when the sum of the border numbers
is the smallest, namely $3$. In this case, $f_{n}(1, 2)$ is
$(n - 2)!$. Indeed, we need to put the tallest building on the left
and the second tallest on the right. After that we can permute the
leftover buildings any way we want.

Second, we want to discuss the case when the sum of the border numbers is the largest: $n+1$. In this case $f_{n}(a,n+1-a)$ is $\binom{n-1}{a-1}$. Indeed, the position of the tallest building is uniquely defined --- it has to take the $a$-th spot from the left. After that we can pick a set of $a-1$ buildings that go to the left from the tallest building and the row is uniquely defined by this set.

\section{Skyscraper Numbers}\label{sec:numbers}

\subsection{One-sided skyscraper numbers}

Before going further let us see what happens if only one of the maxima is given. Let us denote by $c(n, a)$ the number of permutations of $n$ buildings so that you can see $a$ buildings from the left. If we put the shortest building on the left then the leftover buildings need to be arrange so that you can see $a-1$ of them. If the shortest building is not on the left, then it can be in any of the $n-1$ places and we still need to rearrange the leftover buildings so that we can see $a$ of them. We just proved that the function $c(n, a)$ satisfies the recurrence
\[
c(n, a)=c(n - 1, a-1) + (n-1)c(n-1, a) = \sum_{k=1}^{n-1}\frac{(n-1)!}{(n-k-1)!}c(n - k, a-1).
\]
Actually $c(n, a)$ is a well-known function.  The numbers
$c(n, a)$ are called \textit{unsigned Stirling numbers of the first
kind} (see sequence A132393 in the Online Encyclopedia of Integer Sequences \cite{OEIS}); not only
do they count permutations with a given number of left-to-right (or
right-to-left) maxima, but they also count permutations with a given
number of cycles, and they appear as the coefficients in the product
$(x + 1)(x + 2)(x + 3)\ldots (x + n)$, among other places; see for example \cite[Chapter 1]{EC1}.

\subsection{The tallest building is on the left}

We are now equipped to calculate $f_{n}(1, b)$. The tallest building must be on the left, and the rest could be arranged so that, in addition to the tallest building, $b-1$ more buildings are seen from the right. That is, $f_{n}(1, b) = c(n-1, b-1)$.

The non-zero values of $f_{n}(1, b)$ are collected in Table~\ref{fn1b}.

\begin{table}
\begin{center}
    \begin{tabular}{|c|r|r|r|r|r|r|}
        \hline
        ~     & $b$=2 & $b$=3 & $b$=4 & $b$=5 & $b$=6 & $b$=7 \\ \hline
        $n$=2 & 1     & ~     & ~     & ~     & ~     & ~     \\ \hline
        $n$=3 & 1     & 1     & ~     & ~     & ~     & ~     \\ \hline
        $n$=4 & 2     & 3     & 1     & ~     & ~     & ~     \\ \hline
        $n$=5 & 6     & 11    & 6     & 1     & ~     & ~     \\ \hline
        $n$=6 & 24    & 50    & 35    & 10    & 1     & ~     \\ \hline
        $n$=7 & 120   & 274   & 225   & 85    & 15    & 1     \\
        \hline
    \end{tabular}
\caption{The Stirling numbers $c(n - 1, b - 1) = f_{n}(1, b)$}
\label{fn1b}
\end{center}
\end{table}

\subsection{Skyscraper numbers}

Now we have everything we need to consider the general case.  In any
permutation of length $n$, the left-to-right maxima consist of $n$ and all
left-to-right maxima that lie to its left; similarly, the
right-to-left maxima consist of $n$ and all the right-to-left maxima to
its right. We can take any permutation counted by $f_{n}(a, b)$
and split it into two parts: if the value $n$ is in position $k + 1$ for some $0 \leq k \leq n-1$, the first $k$ values form a permutation with $a - 1$ left-to-right maxima and the last $n - k - 1$ values form a permutation with $b - 1$ right-to-left maxima, and there are no other restrictions. Thus,
\[
f_n(a,b)=\sum_{k=0}^{n-1} \binom{n-1}{k} c(k, a-1) \cdot c(n-k-1, b-1).
\]

Table~\ref{f7ab} shows $f_{7}(a,b)$, of which we already calculated the first row:

\begin{table}[ht]
\begin{center}
    \begin{tabular}{|c|r|r|r|r|r|r|r|}
        \hline
        ~     & $b=1$ & $b=2$ & $b=3$ & $b=4$ & $b=5$ & $b=6$ & $b=7$ \\ \hline
        $a=1$ & 0     & 120   & 274   & 225   & 85    & 15    & 1     \\ \hline
        $a=2$ & 120   & 548   & 675   & 340   & 75    & 6     & 0     \\ \hline
        $a=3$ & 274   & 675   & 510   & 150   & 15    & 0     & 0     \\ \hline
        $a=4$ & 225   & 340   & 150   & 20    & 0     & 0     & 0     \\ \hline
        $a=5$ & 85    & 75    & 15    & 0     & 0     & 0     & 0     \\ \hline
        $a=6$ & 15    & 6     & 0     & 0     & 0     & 0     & 0     \\ \hline
        $a=7$ & 1     & 0     & 0     & 0     & 0     & 0     & 0     \\
        \hline
    \end{tabular}
\caption {The skyscraper numbers $f_{7}(a,b)$}
\label{f7ab}
\end{center}
\end{table}

\section{The Skyscraper Sequence}\label{sec:sequence}

We see that the first two rows of the puzzle above (see Figure~\ref{fig:puzzle}) have the same pair of numbers, namely 2 and 3, outside.  If we ignore all other constraints there are $675$ ways to fill in each of the first two rows. Number 675 is the largest number in the Table~\ref{f7ab}. We can say that these two rows of the puzzle are the most difficult to fill in: the pair of numbers 2 and 3 is the least restrictive. Given $n$, we call such a pair the \textit{maximizing pair}. 

By the way, the sequence of the number of ways to fill in the most difficult row for $n$ from 1 to 10 is: 1, 1, 2, 6, 22, 105, 675, 4872, 40614, 403704, 4342080, 50457000, 31548456, 8484089328, 121882518576, 1865935562400, 30341974222944, 522466493255424, 9499883854364928, 181927524046316544. This sequence is now sequence A218531 in the OEIS \cite{OEIS}. 

The maximizing pairs $(a,b)$ are (1, 1), (1, 2), (2, 2), (2, 2), (2, 2), (2, 3), (2, 3), (2, 3), (3, 3), (3, 3), (3, 3), (3, 3), (3, 3), {3, 3), (3, 3), (3, 3), (3, 3), (3, 3), (3, 3), (3, 3), (3, 4), (3, 4), (3, 4), (3, 4), (3, 4), (3, 4), (4, 4), (4, 4), (4, 4), (4, 4). One may notice that the numbers in these pairs do not differ by more than 1. We will prove that in Theorem~\ref{thm:maximizingpairs}.

The actual skyscraper puzzles are designed so that they have a unique solution. It is the interplay between rows and columns that allows to reduce the number of overall solutions to one.

\section{Properties}\label{sec:properties}

If you look at the antidiagonal in Table~\ref{f7ab} you can notice that the numbers there are exactly a row of a Pascal triangle. Moreover, if we rescale other lines that are parallel to the antidiagonal by their $\gcd$ we get exactly the Pascal's Triangle! We combine these observations in the following lemma.

\begin{lemma}
$f_n(a, b) = f_n(a + b - 1, 1) \cdot \binom{a + b - 2}{a - 1}$.
\end{lemma}

\begin{proof}
The number $f_n(a + b - 1, 1)$ counts the number of ways to arrange skyscrapers so that the tallest building is on the right and the permutation has $a+b-1$ left-to-right maxima. For every skyscraper that is not the tallest building consider a group of buildings that are immediately after it to the right and that are located before the next left-to-right maximum. Choose $a-1$ left-to-right maxima out of them and move them together with their groups to the right and put them in the reverse order.
\end{proof}

We want to add a different proof of this lemma through Stirling numbers.

\begin{proof}
We want to show that
$f_{n}(a,b) = \sum_{k = 0}^{n-1} \binom{n-1}{k} \cdot c(k, a-1) \cdot c(n - k-1, b-1)$
is equal to $\binom{a + b-2}{a-1} \cdot c(n, a + b-2)$.
We can do this using the fact that $c(n, a + b-2)$ counts also permutations
of length $n-1$ with $a + b-2$ cycles.  Let us count permutations of length $n-1$
with $a + b-2$ cycles, of which $a-1$ are colored red and $b-1$ are colored blue,
in two different ways: on one hand, we may first choose a permutation
of length $n-1$ with $a + b-2$ cycles, then choose $a-1$ of these cycles to be
red, in $\binom{a + b-2}{a-1} \cdot c(n-1, a + b-2)$ ways.  On the other hand, for any
$k$ between $0$ and $n-1$ we may first choose $k$ elements to be in red cycles,
then make a permutation on these $k$ entries with exactly $a-1$ cycles, and
make a permutation with exactly $b-1$ cycles on the remaining $n - k-1$
entries.  Thus we have in total $\sum_{k = 0}^{n-1} \binom{n-1}{k} \cdot c(k, a-1)
\cdot c(n - k-1, b-1)$ ways in this case.  So the two things are equal, as
claimed.
\end{proof}

Now we are ready to prove the promised theorem about maximizing pairs for the skyscraper sequence.

\begin{thm}\label{thm:maximizingpairs}
The numbers in the maximizing pairs do not differ by more than $1$.
\end{thm}

\begin{proof}
The maximum number in a table corresponding to the given $n$ is the maximum number in its line parallel to antidiagonal. As each diaginal is proportionate to a row in the Pascal's triangle, the maximum can be found in its middle.
\end{proof}

Now let us sum all numbers in each row of Table~\ref{f7ab}. We get the following sums: 720, 1764, 1624, 735, 175, 21, 1. This numbers bring us back to the sequence A132393 in the OEIS \cite{OEIS}. These numbers are a row in the triangle of unsigned Stirling numbers of the first kind. This brings us to the next lemma:

\begin{lemma}
$\sum_b f_n(a,b) = f_{n+1}(a+1,1)$.
\end{lemma}

\begin{proof}
Consider a row of length $n+1$ with $a+1$ left-to-right maxima and 1 right-to-left maximum. In this configuration the tallest building is on the right. If we remove the tallest building, the we get a row of length $n$ with $a$ left-to-right maxima and an unknown number of right-to-left maxima.
\end{proof}

Skyscraper numbers are fun to play with. We are sure that they are hiding many more secrets.

\bigskip
\hrule
\bigskip

\noindent 2010 {\it Mathematics Subject Classification}: Primary 11B83; Secondary 11Y55.

\noindent \emph{Keywords: } sequences, skyscraper puzzle, Stirling numbers.

\bigskip
\hrule
\bigskip

\noindent
(Mentions A132393, A218531.)

\bigskip
\hrule
\bigskip

\end{document}